\documentclass[11pt]{amsart}
\usepackage{mathrsfs}
\usepackage{amsfonts}
\usepackage{amsthm}
\usepackage{amssymb}
\usepackage{latexsym}
\usepackage{tikz}
\usetikzlibrary{calc,shapes.arrows,decorations.pathreplacing, patterns}
\newtheorem{theorem}{Theorem}
\newtheorem{lemma}{Lemma}
\newtheorem{corollary}{Corollary}

\newtheorem*{remarks}{Remarks}

\newtheorem{proposition}{Proposition}

\begin{document}

\title[Complex Green operator: symmetry, percolation, interpolation]{Estimates for the complex Green operator: symmetry, percolation, and interpolation}

\author{S\'{e}verine Biard}
\author{Emil J.~Straube}

\address{Department of Mathematics, School of Engineering and Natural Sciences, University of Iceland, IS-107 Reykjav\'{i}k, Iceland}
\email{biard@hi.is}
\address{Department of Mathematics Texas A\&M University College Station, Texas, 77843}
\email{straube@math.tamu.edu}

\thanks{Supported in part by Qatar National Research Fund Grant NPRP 7-511-1-98 .}

\date{August 8, 2017 (revision)}

\begin{abstract}
Let $M$ be a pseudoconvex, oriented, bounded and closed CR-submanifold of $\mathbb{C}^{n}$ of hypersurface type. We show that Sobolev estimates for the complex Green operator hold  simultaneously for forms of symmetric bidegrees, that is, they hold for $(p,q)$--forms if and only if they hold for $(m-p,m-1-q)$--forms. Here $m$ equals the CR-dimension of $M$ plus one. Symmetries of this type are known to hold for compactness estimates. We further show that with the usual microlocalization, compactness estimates for the positive part percolate up the complex, i.e. if they hold for $(p,q)$--forms, they also hold for $(p,q+1)$--forms. Similarly, compactness estimates for the negative part percolate down the complex. As a result, if the complex Green operator is compact on $(p,q_{1})$--forms and on $(p,q_{2})$--forms ($q_{1}\leq q_{2}$), then it is compact on $(p,q)$--forms for $q_{1}\leq q\leq q_{2}$. It is interesting to contrast this behavior of the complex Green operator with that of the $\overline{\partial}$-Neumann operator on a pseudoconvex domain.
\end{abstract}

\maketitle

\section{Introduction}

A CR-submanifold $M$ of $\mathbb{C}^{n}$ is of hypersurface type, if the real codimension of the complex tangent space inside the real tangent space is one. We will also assume that $M$ is compact, closed, and orientable. We denote the complex dimension of the complex tangent space, i.e. the CR-dimension of $M$, by $(m-1)$ (thus defining $m$; $(m-1)$ is chosen in analogy to a hypersurface in $\mathbb{C}^{m}$). The $L^{2}$-Sobolev theory for the tangential Cauchy Riemann operator and the associated complex Green operator on these CR-manifolds is now at a comparable level of development to that of the $L^{2}$-Sobolev theory of the $\overline{\partial}$-Neumann operator on pseudoconvex domains, as far as various sufficient conditions for estimates (subelliptic, compactness, Sobolev) are concerned. Some of these results are surprisingly recent; \cite{StraubeZeytuncu15}, see \cite{BiardStraube16} for a survey. 

Suppose now that estimates are known for $(p,q)$--forms for \emph{some} pair $(p,q)$. Do these estimates imply corresponding estimates for forms of other bidegrees? Because the $\overline{\partial}$-Neumann problem involves boundary conditions, while the complex Green operator does not, the answer to this question differs for the two problems. For example, compactness and subellipticity for the $\overline{\partial}$-Neumann operator percolate up the $\overline{\partial}$-complex: if these estimates hold for $(p,q)$--forms, they also hold for $(p,q+1)$--forms, see for example \cite{Straube10a}, Proposition 4.5, and the references given there for the original sources. This property of the $\overline{\partial}$-Neumann operator fails for the complex Green operator (see section \ref{est-percolation}). On the other hand, compactness and subellipticity for the complex Green operator are known to hold simultaneously at symmetric bidegrees $(p,q)$ and $(p,m-1-q)$ (\cite{Kohn81}, Proposition on page 255, \cite{Koenig04}, page 289, \cite{BiardStraube16}, Lemma 8). This property manifestly fails for the $\overline{\partial}$-Neumann operator, for example in light of the characterization of compactness in the $\overline{\partial}$-Neumann problem on convex domains given in \cite{FuStraube98} (namely, compactness holds for $(p,q)$--forms if and only if the boundary of the domain contains no germ of a $q$-dimensional affine complex submanifold).

These observations notwithstanding, there is percolation of estimates for the $\overline{\partial}_{M}$ complex, but it is more subtle. If $\mathcal{P}^{+}$, $\mathcal{P}^{-}$, and $\mathcal{P}^{0}$ denote the usual microlocalizations, and $G_{(p,q)}$ denotes the complex Green operator for $(p,q)$--forms, then compactness of $\mathcal{P}^{+}G_{(p,q)}$ percolates \emph{up} the complex, while the analogous estimates for $\mathcal{P}^{-}G_{(p,q)}$ percolate \emph{down}. Because $\mathcal{P}^{0}G_{(p,q)}$ is elliptic, the full estimates then interpolate: if the
complex Green operator is compact on $(p,q_{1})$--forms and on $(p,q_{2})$--forms ($q_{1}\leq q_{2}$), then it is compact on $(p,q)$--forms for $q_{1}\leq q\leq q_{2}$.\footnote{After this paper was posted and submitted for publication, we became aware of  \cite{Khanh16}, where a closely related result is shown in the case where $M$ is an actual hypersurface. We thank Ken Koenig for pointing out this reference to us.}


The simultaneous validity of compactness or subellipticity for the complex Green operator mentioned above has been shown with the help of a somewhat ad hoc local `Hodge-$\star$--like' operator that, unlike the actual Hodge-$*$, maps $(p,q)$--forms to $(p, m-1-q)$--forms and intertwines $\overline{\partial}_{M}$ and $\overline{\partial}_{M}^{*}$ modulo terms of order zero (\cite{Kohn81, Koenig04, BiardStraube16}). These terms are of no consequence for compactness and subellipticity, they can simply be absorbed. However, when one considers Sobolev estimates, these error terms matter. To avoid them, we use the natural pairing between a $(p,q)$--form $\alpha$ and an $(m-p, m-1-q)$--form $\beta$ on $M$ given by $\int_{M}\alpha\wedge\beta$, and the associated Hodge-$\star$ operator to define a conjugate linear operator
$A_{p,q}: L^{2}_{(p,q)}(M) \rightarrow L^{2}_{(m-p,m-1-q)}(M)$. Accommodating a technicality requires adjusting the metric on forms on $M$ and taking the $\star$--operator with respect to this adjusted metric. The operators so defined intertwine $\overline{\partial}_{M}$ and $\overline{\partial}_{M}^{*}$ without error terms. As a result, Sobolev estimates for the complex Green operator (and many others) hold for $(p,q)$--forms if and only if they hold for $(m-p,m-1-q)$--forms.

\section{Preliminaries and notation}\label{pre}

In this section, we consider a smooth compact and orientable CR-submanifold $M$ in $\mathbb{C}^n$, without boundary. Define $m$ via $dim_{\mathbb{C}}T^{\mathbb{C}}M = (m-1)$, where $T^{\mathbb{C}}_{P}M$ denotes the complex tangent space at $P$, i.e. $T_P M\cap JT_P M$, where $T_P M$ is the real tangent space to $M$ and $J$ the complex structure map on $\mathbb{C}^{n}$ (i.e. multiplication by $i$). This dimension is independent of $P$.
$M$ is said to be of hypersurface type if, at each point $P\in M$, $T^{\mathbb{C}}_P M$ has real codimension one in $T_P M$. A vector field 
$X(z)=\sum_{j=1}^{n}a_{j}(z)\partial/\partial z_{j}$ (on an open set of $\mathbb{C}^{n}$ or of $M$) is called of type $(1,0)$, while a field $Y(z)=\sum_{j=1}^{n}b_{j}(z)\partial/\partial\overline{z_{j}}$ is of type $(0,1)$, as usual. $X$ is tangential to $M$ if and only if $(a_{1}(z), \hdots, a_{n}(z)) \in T^{\mathbb{C}}_{z}M$, for all $z$; similarly, $Y(z)$ is tangential if and only if $(\overline{b_{1}(z)}, \hdots, \overline{b_{n}(z)}) \in T^{\mathbb{C}}_{z}M$, for all $z$. We say that $X\in T^{1,0}M$, $Y\in T^{0,1}M$ ($T^{1,0}M$ and $T^{\mathbb{C}}M$ are thus naturally isomorphic). For detailed information on CR-(sub)manifolds, the reader may consult \cite{Boggess91, BER99}.

Because $M$ is orientable, there exits a purely imaginary vector field $T$ on $M$ of unit length that is orthogonal to $T^{\mathbb{C}}M$ at all points. Let $\eta$ be the form dual to $T$, that is $\eta(T)\equiv 1$, and $\eta \equiv 0$ on $T^{1,0}M \oplus T^{0,1}M$. Denote by $L_{m}$ the vector field $L_{m}:=(1/\sqrt{2})(T-iJT)$ defined on $M$; $L_{m}$ is of type $(1,0)$ and has length one. Near a point $P\in M$, choose an orthonormal basis $\{L_{1}, \hdots, L_{(m-1)}\}$ of $T^{1,0}M$. Choose $(1,0)$-forms $\{\omega_{1}, \hdots, \omega_{m}\}$ that at each point vanish on $\{L_{1}, \hdots, L_{m}\}^{\perp}$ and so that $\omega_{k}(L_{j})=\delta_{k,j}$. These are the usual local frames. Note that when we restrict $\omega_{m}$ to $M$ as a form, this restriction does not equal $\eta$; rather, we have $\omega_{m}|_{M}=(1/\sqrt{2})\eta$ (see for example \cite{Range86}, ch. III.3 for a discussion of the Hermitian structure on $\mathbb{C}^{n}$ that pays attention to norms of the $dz_{j}$, etc.).

The space of $(p,q)$--forms on $M$ at $P$,
$\Lambda^{p,q}T^*_{P}M$, is defined as those 
forms in $\Lambda^{p,q} T^*_{P}\mathbb{C}^n$ that have the form
\begin{equation}\label{p-q-forms}
u=\sideset{}{'}\sum_{{\vert I\vert=p, \vert J\vert =q}} u_{IJ}\omega_I(P)\wedge \overline{\omega_J}(P), \; I\subseteq\lbrace{1,\dots, m}\rbrace,\; J\subseteq \lbrace{1,\dots, m-1}\rbrace.
\end{equation}
The notation $\sum'$ indicates summation over strictly increasing multi-indices. This definition is independent of the choice of orthonormal basis $\{L_{1},\hdots, L_{(m-1)}\}$ of $T^{1,0}M$ near $P$ ($L_{m}$ is defined globally, and therefore, so is $\omega_{m}$).
\footnote{When $p\neq 0$, and $M$ is not generic, this definition differs from that given in \cite{Boggess91}. However, in this situation, the definition in \cite{Boggess91} allows for $(p,q)$-forms whose restrictions to $M$ vanish, so that the resulting complex need not be isomorphic to the intrinsically defined complex. In fact, section 8.3 in \cite{Boggess91}, with the extrinsic definition used there, requires the assumption that $M$ be generic \cite{Boggess16}.} 

The (extrinsic) tangential Cauchy--Riemann operator is now defined in the usual way.  Locally, represent a $(p,q)$--form as in \eqref{p-q-forms}. Extend $u$ coefficientwise to a form $\widetilde{u}$ defined in a full neighborhood in $\mathbb{C}^{n}$ (note that the local frame `lives' in such a full neighborhood). Then
\begin{equation}\label{d-bar-M}
\bar\partial_M u=(\bar\partial \tilde{u})_{t_M} 
\end{equation}
where $t_M: \Lambda^{p,q}T^*_{P}\mathbb{C}^n\rightarrow \Lambda^{p,q}T^*_{P}M$ is the orthogonal projection, for $P\in M$ (that is $t_{M}$ gives the tangential part of a form). This definition is independent of the local frame and/or the extension chosen, so that $\overline{\partial}_{M}$ is well defined by \eqref{d-bar-M}. We also have $\bar\partial_M\circ\bar\partial_M=0$; this property is inherited from the $\overline{\partial}$-complex on $\mathbb{C}^{n}$. It is useful to have the following expression for $\overline{\partial}_{M}$ in a local frame:
\begin{equation}\label{localexpress}
\bar\partial_M u=\sum_{k=1}^{m-1}\sideset{}{'}\sum_{\vert I\vert =p, \vert J\vert =q} \bar{L}_k(u_{IJ})\bar{\omega}_k\wedge \omega_I\wedge \overline{\omega_J} + \text{terms of order zero}\;. 
\end{equation}
Here, terms of order zero means terms where the coefficients of $u$ are not differentiated.
We refer to \cite{Boggess91, BiardStraube16} for more details.

The pointwise inner product between $(p,q)$--forms at $P\in M$,
\begin{equation}\label{pointwise}
 <u,v>\;\; = \sideset{}{'}\sum_{{\vert I\vert=p, \vert J\vert =q}}u_{IJ}\overline{v_{IJ}}
\end{equation}
is independent of the choice of the local othonormal frame. 
It provides an $L^2$-inner product on $M$ by integrating against the (Euclidean) volume element on $M$, as usual:
\begin{equation}\label{L2product}
(u,v)_{L^2_{(p,q)}(M)} =\int_M <u(z),v(z)> dV_M(z)\,.
\end{equation}
We denote by $L^2_{(p,q)}(M)$, $0\leq p\leq m, 0\leq q\leq (m-1)$, the completion of $\Lambda^{p,q}T^{*}M$ under the norm induced by this inner product . 


$\overline{\partial}_{M}$ extends to an unbounded operator $\bar\partial_{M}:L^2_{(p,q)}(M)\rightarrow  L^2_{(p,q+1)}(M)$ acting in the sense of distributions, with the maximal domain of definition. That is, we set $dom(\bar\partial_M)=\lbrace{u\in L^2_{(p,q)}(M)\mid \bar\partial_M u\in L^2_{(p,q+1)}(M)}\rbrace$, where $\overline{\partial}$ acts in a local frame as in \eqref{localexpress}. Whether or not the resulting coefficients are in $L^{2}$ locally does not depend on the choice of the frame.
As a closed and densely defined operator on $L^2_{(p,q)}(M)$, $1\leq q\leq m-1$, $\bar\partial_M$ has a Hilbert space adjoint, denoted by $\bar\partial_M^*$. In a local frame, integration by parts gives
\begin{equation}\label{adjoint}
\bar\partial^*_M u=-\sum_{j=1}^{m-1}\sideset{}{'}\sum_{\vert I\vert =p, \vert K\vert =q-1} L_j( u_{IjK})\omega_I\wedge \overline{\omega_K} +\text{ terms of order zero} \; .
\end{equation}
A crucial fact is that $\overline{\partial}_{M}$, hence $\overline{\partial}_{M}^{*}$, have closed range, see \cite{Baracco1, Baracco2, Baracco3, BiardStraube16}.
A useful technical tool is the fact that
\begin{equation}\label{dense}
 C^{\infty}_{(p,q)}(M)\cap dom(\overline{\partial}_{M})\cap dom(\overline{\partial}_{M}^{*})\;
 \text{is dense in}\; dom(\overline{\partial}_{M})\cap dom(\overline{\partial}_{M}^{*})\;;
\end{equation}
that is, the smooth forms are dense in $dom(\overline{\partial}_{M})\cap dom(\overline{\partial}_{M}^{*})$. This is a standard consequence of the Friedrichs Lemma (see for example \cite{ChenShaw01}, Appendix D); in contrast to the $\overline{\partial}$-complex, there are no boundary conditions for the $\overline{\partial}_{M}$--complex that require extra care with the regularization.

Finally, let $1\leq q\leq (m-2)$. The complex Laplacian on $L^2_{(p,q)}(M)$, denoted by $\square_{(p,q)}$, is defined as $\bar\partial_M\bar\partial_M^*+\bar\partial^*_M\bar\partial_M$; its domain $dom(\square_{(p,q)})$ is understood to be the set of forms where this expression makes sense. This operator is the unique self-adjoint operator associated to the quadratic form 
$Q_{p,q}(u,u)=(\bar\partial_M u, \bar\partial_M u)_{L^2_{(p,q+1)}(M)} + (\bar\partial_M^* u, \bar\partial_M^* u)_{L^2_{(p,q-1)}(M)}$, via
\begin{equation}\label{quadform}
Q_{p,q}(u,u) = (\Box_{p,q}u,u)_{L^{2}_{(p,q)}(M)} \; , \;u\in dom(\square_{p,q})\; .
\end{equation}
We denote $ker(\square_{(p,q)})=\mathcal{H}_{(p,q)}(M)$, the harmonic $(p,q)$-forms on $M$ with $L^2$-coefficients. The dimension of $\mathcal{H}_{(p,q)}(M)$ is known to be finite when $1\leq q\leq (m-2)$ (\cite{Nicoara06, HarringtonRaich10}), but it need not be zero, even for strictly pseudoconvex $M$; see the discussion in \cite{StraubeZeytuncu15} pp. 1076--1077, and following Corollary 1 there. That the dimension of $\mathcal{H}_{(p,q)}(M)$ is finite is reflected in a version of the basic $L^{2}$--estimate where the norm of the harmonic component of a form $u$ is replaced by $\|u\|_{-1}$ (see \cite{StraubeZeytuncu15}, estimate 7,\cite{BiardStraube16}, Lemma 5):
\begin{multline}\label{basicL2}
\|u\|_{L^{2}_{(p,q)}(M)}^{2} \lesssim \|\overline{\partial}_{M}u\|_{L^{2}_{(p,q+1)}(M)}^{2} + \|\overline{\partial}_{M}^{*}u\|_{L^{2}_{(p,q-1)}(M)}^{2} + \|u\|_{-1}^{2} \;,\\
 u \in dom(\overline{\partial}_{M}) \cap dom(\overline{\partial}_{M}^{*})\;,\;0\leq p\leq m\;,\;1\leq q\leq (m-2)\;.
\end{multline}
Because the range of $\overline{\partial}_{M}$ is closed, so is that of $\Box$. Also, $\Box_{(p,q)}$ maps $\mathcal{H}_{(p,q)}(M)^{\perp}$ onto itself.
The complex Green operator, $G_{p,q}$ is the inverse operator of the restriction of $\square_{(p,q)}$ to $\mathcal{H}_{(p,q)}(M)^{\perp}$.
It is convenient to extend it to all of $L^{2}_{(p,q)}(M)$ by setting it equal to zero on $\mathcal{H}_{(p,q)}(M)$. $G_{p.q}$ is a bounded self--adjoint operator. A detailed discussion of these matters may be found in \cite{BiardStraube16, ChenShaw01} (partly for $(0,q)$--forms, but the arguments are the same for $(p,q)$--forms).\footnote{The complex Green operators can also be defined for bidegrees $(p,0)$ and $(p,m-1)$ in much the same way as the $\overline{\partial}$-Neumann operators $N_{p,0}$ are defined, and appropriate statements can easily be given. This does not add to the main thrust of this paper; accordingly, we do not consider these cases.}

\section{Estimates on symmetric bidegrees}\label{duality}

We first construct a conjugate linear operator $A_{(p,q)}$ mapping $(p,q)$-forms to $(m-p,m-q-1)$-forms on $M$ that intertwines $\overline{\partial}_{M}$ and $\overline{\partial}_{M}^{*}$. Our arguments involve on the one hand integration by parts in an integral of a wedge product of forms, on the other hand integration by parts (moving $\overline{\partial}_{M}$ to the other side as $\overline{\partial}_{M}^{*}$) in the inner product \eqref{L2product} between forms. Thus the Hodge-$\star$ operator arises naturally. Note that the pointwise inner product \eqref{pointwise} between two forms in $\Lambda^{p,q}T^{*}_{P}M$ does not necessarily agree with the inner product of their restrictions to $M$ at $P$. The reason is that the unit form $\omega_{m}(P) \in \Lambda^{1,0}T^{*}_{P}M$ restricts to $(1/\sqrt{2})\eta \in \mathbb{C}T_{P}^{*}M$, a form of norm $(1/\sqrt{2})$. In order to rectify this situation, we change the metric on $\mathbb{C}TM$, hence on $\mathbb{C}T^{*}M$ by declaring, at each point $P\in M$, $\{\omega_{1}, \hdots, \omega_{(m-1)}, \overline{\omega_{1}}, \hdots, \overline{\omega_{(m-1)}}, (1/\sqrt{2})\eta\}$ to be an orthonormal basis. In other words, we rescale in the direction of $\eta$ by a factor of $\sqrt{2}$ (equivalently, by a factor of $1/\sqrt{2}$ in the direction of $T$). When we equip $M$ with this new Riemannian structure, the restriction of forms in $\Lambda^{p,q}T^{*}_{P}M$ to $M$ (restriction as forms) becomes an isometry (at the point $P$). We use $\widetilde{\star}$, $<\;,\;>_{\sim}$, and $d\widetilde{V}$ to denote, respectively, the Hodge-$\star$ operator, the pointwise inner product on forms, and the volume element on $M$ with respect to this new Riemannian structure. All properties of the Hodge-$\star$ operator that we will use can be found in \cite{Range86}, section III.3.4 and/or in \cite{Morita01}, section 4.1 (c).

We now define the conjugate linear operator $A_{p,q}: L^{2}_{(p,q)}(M) \rightarrow L^{2}_{(m-p, m-1-q)}(M)$ via
\begin{multline}\label{Apq}
\left(v, A_{p,q}u\right) := \sqrt{2}\int_{M}u\wedge v \;, \;\;u\in L^{2}_{(p,q)}(M),\; v\in L^{2}_{(m-p, m-1-q)}(M) \; , \\ 
0\leq p\leq m\;,\,0\leq q\leq (m-1)\;.
\end{multline}
This definition is analogous to the one in the appendix of \cite{RaichStraube08}.
It will be convenient to express $A_{p,q}$ with the help of $\widetilde{\star}$. We have

\begin{multline}\label{Apq-star}
\left(v, A_{p,q}u\right) = \sqrt{2}\int_{M}u\wedge v = \sqrt{2}\int_{M}(\widetilde{\star}\widetilde{\star}(u|_{M}))\wedge v = \sqrt{2}\int\overline{\widetilde{\star}\left(\widetilde{\star}\overline{(u|_{M})}\right)}\wedge v \\
= \sqrt{2}\int_{M}v\wedge\overline{\widetilde{\star}\left(\widetilde{\star}\overline{(u|_{M})}\right)}
= \sqrt{2}\int_{M}<v|_{M}, \,\widetilde{\star}\overline{(u|_{M})}>_{\sim}d\widetilde{V} 
= \int_{M}<v,\,\widetilde{\star}\overline{(u|_{M}}))>\,dV \;.
\end{multline}
Therefore,
\begin{equation}\label{Apq-star2}
 A_{p,q}u = \widetilde{\star}\,\left(\overline{u|_{M}}\right) \;,\; u\in L^{2}_{(p,q)}(M)\; ,
\end{equation}
in the sense that $A_{p,q}u$ equals the unique form in $L^{2}_{(m-p, m-1-q)}(M)$ whose restriction to $M$ equals $\widetilde{\star}(\overline{u|_{M}})$ (that is, $\eta$ is replaced by $\omega_{m}$). The properties of $A_{p,q}$ that we will need are summarized in the following proposition.

\begin{proposition}\label{Apqproperties}
Let $0\leq p\leq m$, $0\leq q\leq (m-1)$. Then
\begin{eqnarray}
 A_{p,q}: L^{2}_{(p,q)}(M) \rightarrow L^{2}_{(m-p,m-1-q)}(M) \;\,\mbox{is an isometry}\;, \label{isometry}
\end{eqnarray}
\begin{eqnarray}
A_{m-p,m-q-1}A_{p,q}u= u, \quad \forall u\in L^{2}_{(p,q)}(M)\;.\label{Id}
\end{eqnarray}
Let $0\leq p\leq m$, $1\leq q\leq (m-2)$. Then
 \begin{eqnarray}
\bar\partial_M A_{p,q}u= (-1)^{p+q}A_{p, q-1}\bar\partial_M^* u, \quad \forall u\in 
dom(\overline{\partial}_{M}^{*}) \subseteq L^{2}_{(p,q)}(M)\;,\label{Adbar}\end{eqnarray}
\begin{eqnarray}
A_{p,q}\bar\partial_M u= (-1)^{p+q}\bar\partial_M^*A_{p,q-1} u, \quad \forall u\in dom(\overline{\partial}_{M}) \subseteq L^{2}_{(p,q-1)}(M)\;,\label{dbarA}
\end{eqnarray}
\begin{eqnarray}
A_{p,q}\Box_{p,q}u = \Box_{m-p, m-1-q}A_{p,q}u, \quad \forall u\in dom(\Box_{p,q})\subseteq L^{2}_{(p,q)}(M) \; , \label{commutes}
\end{eqnarray}
\begin{eqnarray}
A_{p,q}G_{p,q} = G_{m-p,m-1-q}A_{p,q} \; , \label{commutes2}
\end{eqnarray}
\begin{eqnarray}
A_{p,q}(\mathcal{H}_{(p,q)}(M)) = \mathcal{H}_{(m-p,m-1-q)}(M) \; . \label{harmonicforms}
\end{eqnarray}
\end{proposition}
It is part of the proposition that in equations \eqref{Adbar} -- \eqref{commutes}, if $u$ is in the domain of $\overline{\partial}_{M}^{*}$, $\overline{\partial}_{M}$, and $\Box_{p,q}$, respectively, then $A_{p,q}u$ and $A_{p,q-1}u$ are in the appropriate domain of the operator on the other side of the equation.
\begin{proof}
It suffices to prove all the statements for smooth forms; they are dense in $L^{2}_{(p,q)}(M)$ and in the graph norms of both $\overline{\partial}_{M}$ and $\overline{\partial}_{M}^{*}$ (this is immediate from a standard mollifier argument and the Friedrichs Lemma, as there are no boundary conditions to take into account).

\eqref{isometry} and \eqref{Id} are immediate from \eqref{Apq-star2} and the fact that $\widetilde{\star}$ is an isometry in the modified metric on $M$, and that $\widetilde{\star}\widetilde{\star}u = u$ (there is a factor $(-1)^{(p+q)(2m-1-p-q)}$; however, $(p+q)(2m-1-p-q) \equiv 0 \,mod\, 2$).


To verify \eqref{dbarA}, let $u\in \Lambda^{p,q-1}T^{*}M$, $v\in \Lambda^{m-p,m-1-q}T^{*}M$.
Note that 
\begin{equation}
\int_{M}\overline{\partial}_{M}u\wedge v = \int_{M}(\partial_{M}+\overline{\partial}_{M})u\wedge v = \int_{M}du\wedge v 
\end{equation}
($\displaystyle{\int_{M}\partial_{M}u\wedge v =0}$, because at least one of the $\omega_{j}$, $1\leq j\leq m$, will appear twice, or there will be an $\omega_{j}$ with $j>m$; in either case, the integral over $M$ vanishes). Integration by parts therefore gives
\begin{multline}\label{computation2}
 (v, A_{p,q}\overline{\partial}_{M}u)_{L^{2}_{(m-p,m-1-q)}(M)} = \sqrt{2}\int_{M}\overline{\partial}_{M}u\wedge v = (-1)^{p+q}\sqrt{2}\int_{M}u\wedge\overline{\partial}_{M}v \\
 = (-1)^{p+q}\sqrt{2}\int_{M}\widetilde{\star}(\overline{\widetilde{\star}\overline{u|_{M}})}\wedge \overline{\partial}_{M}v 
 = (-1)^{p+q}\sqrt{2}\int_{M}\overline{\partial}_{M}v\wedge \widetilde{\star}(\overline{\widetilde{\star}\overline{u|_{M}})} \; .
\end{multline} 
We have also used that $\widetilde{\star}$ is real, so that $\widetilde{\star}(\overline{\widetilde{\star}\overline{u|_{M}})} = u|_{M}$. Using $\widetilde{\star}$ to mediate between  wedge products and inner products gives
 \begin{multline}\label{computation3}
 \sqrt{2}\int_{M}\overline{\partial}_{M}v\wedge \widetilde{\star}(\overline{\widetilde{\star}\overline{u|_{M}})}
 = \sqrt{2}\int_{M}<\overline{\partial}_{M}v|_{M}, \widetilde{\star}\overline{u|_{M}}>_{\sim}d\widetilde{V} \\
 = (\overline{\partial}_{M}v, \widetilde{\star}\overline{u|_{M}})_{L^{2}_{(m-p,m-q)}(M)} 
 = (v, \overline{\partial}_{M}^{*}A_{p,q-1}u)_{L^{2}_{(m-p, m-1-q)}(M)} \; .
\end{multline}
In the second equality, we use that $\sqrt{2}<\,,\,>_{\sim}d\widetilde{V}=\,<\,,\,>dV$ , as well as \eqref{Apq-star2}. \eqref{computation2} and \eqref{computation3} now imply \eqref{dbarA}.

\eqref{Adbar} follows from \eqref{dbarA} and \eqref{Id}.

Using \eqref{Adbar} and \eqref{dbarA} to move the $A_{r,s}$ operators successively past $\overline{\partial}_{M}$ and $\overline{\partial}_{M}^{*}$ gives \eqref{commutes}; the factor $(-1)^{p+q}$ arises twice and so cancels.

Finally, \eqref{commutes2} and \eqref{harmonicforms} are consequences of \eqref{commutes}.
\end{proof}



Proposition \ref{Apqproperties} immediately gives symmetry of estimates with respect to form levels for the complex Green operator.
We say that $G_{p,q}$ is regular in Sobolev spaces if $\|G_{p,q}u\|_{s} \leq C_{s}\|u\|_{s}$, $s\geq 0$, where $\|\cdot\|_{s}$ denotes the $L^{2}$--Sobolev norm (defined coefficientwise in local charts; this involves choosing a cover of $M$ by charts, but all the resulting norms are equivalent). We say that $G_{p,q}$ is globally regular if it maps ($C^{\infty}$) smooth forms to smooth forms.

\begin{theorem}\label{symmetry}
Let $M$ be a smooth compact pseudoconvex orientable CR-submanifold of $\mathbb{C}^n$ of hypersurface type, of CR-dimension $m-1$. Let $0\leq p\leq m$, $1\leq q\leq (m-2)$. Then $G_{p,q}$ is regular in Sobolev norms (respectively globally regular) if and only if $G_{m-p,m-1-q}$ is.
\end{theorem}
\begin{proof}
The theorem follows from Proposition \ref{Apqproperties}, \eqref{commutes2}, once one observes that $A_{p,q}$ is continuous not only in $L^{2}$, but also in Sobolev norms. The latter fact follows for example from the expression \eqref{Apq-star2} for $A_{p,q}$ and then writing out $\widetilde{\star}$ in local coordinates.\footnote{Obviously, these arguments work for estimates in many other topologies as well (for example, estimates in $L^{p}$--Sobolev norms for $p\neq 2$, etc; see \cite{Koenig04} for H\"{o}lder and $L^{p}$--estimates).}\end{proof}

As to when Sobolev estimates actually do hold when $M$ is as in Theorem \ref{symmetry}, we refer the reader to \cite{BoasStraube91, StraubeZeytuncu15}, as well as to the recent survey \cite{BiardStraube16} and their references.

\section{Percolation of estimates}\label{est-percolation}

Compactness and subellipticity for the $\overline{\partial}$-Neumann operator percolate up the $\overline{\partial}$-complex: if these estimates hold for $(p,q)$--forms, they also hold for $(p,q+1)$--forms, see for example \cite{Straube10a}, Proposition 4.5, and the references given there for the original sources. As mentioned in the introduction, this property of the $\overline{\partial}$-Neumann operator fails for the complex Green operator. Note that taking into account that compactness of the complex Green operator holds (or fails) simultaneously at bidegrees $(p,q)$ and $(p,m-1-q)$ (\cite{Kohn81}, Proposition on page 255, \cite{Koenig04}, page 289, \cite{BiardStraube16}, Lemma 8), a moment's reflection reveals that if compactness were to percolate up the $\overline{\partial}_{M}$--complex, then compactness at \emph{some} level $(p,q)$ would imply compactness at \emph{all} levels $(p, r)$, $1\leq r\leq (m-2)$. This is too good to be true. On the boundary of a smooth bounded convex domain in $\mathbb{C}^{n}$, $G_{0,q}$ is compact if and only if this boundary does not contain complex varieties of dimension $q$ nor of dimension $(n-1-q)$ (\cite{RaichStraube08}, Theorem 1.5). Therefore, if $n\geq 5$, and if the boundary contains an analytic disc, but no higher dimensional complex varieties, then $G_{0,2}, \hdots, G_{0,n-3}$ are compact, while $G_{0,1}$ and $G_{0,n-2}$ are not (and indeed compactness of $G_{0,n-3}$ does not percolate up to $G_{0,n-2}$).

\smallskip

The above characterization of compactness of $G_{0,q}$ on the boundary of a convex domain implies that if $G_{0,q_{1}}$ and $G_{0,q_{2}}$, $q_{1}\leq q_{2}$, are compact, then so is $G_{0,q}$ for $q_{1}\leq q\leq q_{2}$; we refer to this phenomenon as interpolation between bidegrees. This interpolation phenomenon turns out to be true in general. While this is perhaps not surprising, there is more than meets the eye. If $G_{p,q}u = \mathcal{P}^{+}G_{p,q}u + \mathcal{P}^{0}G_{p,q}u + \mathcal{P}^{-}G_{p,q}u$ denotes the usual microlocal split of the form $G_{p,q}u$ (\cite{Kohn85, Kohn02}), then $G_{p,q}$ is compact if and only if $\mathcal{P}^{j}G_{p,q}$ is for $j\in \{+,-,0\}$ (since the $\mathcal{P}^{j}$ are bounded operators on $L^{2}_{(p,q)}(M)$). $\mathcal{P}^{0}G_{p,q}$ is always compact, because of elliptic estimates for $\overline{\partial}_{M}\oplus\overline{\partial}_{M}^{*}$ on that part of the microlocalization, so only $\mathcal{P}^{+}G_{p,q}$ and $\mathcal{P}^{-}G_{p,q}$ are relevant for the question of compactness of $G_{p,q}$. It turns out that compactness for both
$\mathcal{P}^{+}G_{p,q}$ and $\mathcal{P}^{-}G_{p,q}$ does percolate. However, while for $\mathcal{P}^{+}G_{p,q}$, percolation is indeed up the $\overline{\partial}_{M}$--complex, for 
$\mathcal{P}^{-}G_{p,q}$ it is \emph{down} the complex. Of course, interpolation is an immediate corollary: if $G_{p,q}$ is compact at two levels $(p,q_{1})$ and $(p,q_{2})$, $q_{1}\leq q_{2}$, then both $\mathcal{P}^{+}G_{p,q_{j}}$ and $\mathcal{P}^{-}G_{p,q_{j}}$ are compact, $j=1,2$, and percolation (up from $\mathcal{P}^{+}G_{p,q_{1}}$, down from $\mathcal{P}^{-}G_{p,q_{2}}$) implies that at the intermediate form levels $(p,r)$, $q_{1}\leq r\leq q_{2}$, both $\mathcal{P}^{+}G_{p,r}$ and $\mathcal{P}^{-}G_{p,r}$ are compact. Hence so is $G_{p,r}$.

\smallskip

To make these ideas precise, we follow \cite{MunasingheStraube12} in setting up the microlocalizations from \cite{Kohn85, Kohn02}. We first work with forms supported in a fixed open set $U\subset\subset U^{\prime}$ small enough so that the following makes sense. Choose coordinates on $M$ in $U^{\prime}$ of the form $(x_1,\dots, x_{2m-2}, t)$ such that $T=(-i)\frac{\partial}{\partial t}$. Denote the `dual' coordinates in $\mathbb{R}^{2m-1}$ by $(\xi_1,\dots,\xi_{2m-2}, \tau)=(\xi,\tau)$. Next, choose $\chi \in C_0^\infty(U^{\prime})$ with $\chi\equiv 1$ in a neighborhood of $\overline{U}$. On the unit sphere $\lbrace{\Vert \xi\Vert^2+ \tau^2=1}\rbrace$, choose a smooth function $g$, $0\leq g\leq 1$, supported in $\lbrace{\tau>\frac{\Vert \xi\Vert}{2}}\rbrace$, $g\equiv1$ on $\lbrace{\tau\geq \frac{3\Vert \xi\Vert}{4}}\rbrace$.
For $\vert(\xi,\tau)\vert\geq \frac{3}{4}$, set $\chi^+(\xi,\tau)=g(\frac{(\xi,\tau)}{\vert(\xi,\tau)\vert})$ and extend it smoothly to $\vert(\xi,\tau)\vert< \frac{3}{4}$ such that $\chi^+(\xi,\tau)= 0$ on $\lbrace{\vert(\xi,\tau)\vert\leq \frac{1}{2}}\rbrace$. Then, define $\chi^-$ and $\chi^0$ by $\chi^-(\xi,\tau)=\chi^+(-\xi,-\tau)$ and $\chi^0=1-\chi^+ -\chi^-$. 
Finally, denote the Fourier transform on $\mathbb{R}^{2m-1}$ by $\mathcal{F}$. For a $(p,q)$--form $u$,  we set
\begin{equation}\label{microdef}
\mathcal{P}^+u=\chi\mathcal{F}^{-1}\chi^+ \hat{u}\;,\quad \mathcal{P}^-u=\chi\mathcal{F}^{-1}\chi^- \hat{u}\;,\quad \mathcal{P}^0u=\chi\mathcal{F}^{-1}\chi^0 \hat{u}\;,
\end{equation}
where $\hat{u}=\mathcal{F}u$, and the operators act coefficientwise with respect to a fixed 
(chosen) frame $\{\omega_{1}, \hdots, \omega_{m}\}$.\footnote{We assume that the open set $U^{\prime}$ is contained in a special boundary chart.} Then $\mathcal{P}^\pm$ and $\mathcal{P}^0$ also act coefficientwise, as pseudo-differential operators of order zero. Note that $\mathcal{P}^{+}u + \mathcal{P}^{-}u + \mathcal{P}^{0}u = u$.

Cover $M$ with finitely many open sets $U_{j}\subset\subset U_{j}^{\prime}$ as above, $1\leq j\leq l$,  and choose a partition of unity $\{\varphi_{j}\}_{j=1}^{l}$ subordinate to this cover. Then for each $j$, we have the operators $\mathcal{P}^{k}_{j}$, $k=+,-,0$ from the previous paragraph. We set
\begin{equation}\label{micro}
\mathcal{P}^{k}u := \sum_{j=1}^{l}\mathcal{P}^{k}_{j}\varphi_{j}u \;,\;u\in L^{2}_{(p,q)}(M)\;,\;k=+,-,0\;.                                                                                                                                                                                                                                                                                      
                                                                                                                                                                                                                                                                             \end{equation}

Recall that saying that $G_{p,q}$ is compact is the same as saying that the imbedding $j_{p,q}$ of $dom(\overline{\partial}_{M})\cap dom(\overline{\partial}_{M}^{*})\cap 
\mathcal{H}_{(p,q)}(M)^{\perp}$, with the graph norm, into $L^{2}_{(p,q)}(M)$ is compact (\cite{BiardStraube16}, Lemma 6)\footnote{The proof there is for $(0,q)$--forms, but it works equally well for $(p,q)$--forms.}. Compactness of the microlocalizations of the complex Green operators can similarly be expressed in terms of the microlocalizations of $j_{p,q}$. In turn, the latter is equivalent to (a family of) compactness estimates. In the following lemma, $dom(\overline{\partial}_{M})\cap dom(\overline{\partial}_{M}^{*})\cap \mathcal{H}_{(p,q}(M)^{\perp}$ is endowed with the graph norm $\|u\|_{graph} = \|\overline{\partial}_{M}u\| + \|\overline{\partial}_{M}^{*}u\|$, as usual.

\begin{lemma}\label{microcompact}
Let $0\leq p\leq m$, $1\leq q\leq (m-2)$, $k\in \{+,-,0\}$. Then the following are equivalent:

(i) $\mathcal{P}^{k}G_{p,q}$ is compact.

(ii) $\mathcal{P}^{k}j_{p,q}: dom(\overline{\partial}_{M})\cap dom(\overline{\partial}_{M}^{*})\cap \mathcal{H}_{(p,q)}(M)^{\perp} \rightarrow L^{2}_{(p,q)}(M)$ is compact.

(iii) For all $\varepsilon>0$, there is a constant $C_{\varepsilon}>0$ such that 
\begin{equation}\label{microcompest}
 \|\mathcal{P}^{k}u\|^{2} \leq \varepsilon\left(\|\overline{\partial}_{M}u\|^{2} + \|\overline{\partial}_{M}^{*}u\|^{2}\right) + C_{\varepsilon}\|u\|_{{-1}}^{2}\;\;, u\in dom(\overline{\partial}_{M})\cap dom(\overline{\partial}_{M}^{*})\cap \mathcal{H}_{(p,q)}(M)^{\perp}\;.
\end{equation}

(iii)* For all $\varepsilon>0$, there is a constant $C_{\varepsilon}$ such that 
\begin{equation}\label{microcompest2}
 \|\mathcal{P}^{k}u\|^{2} \leq \varepsilon\left(\|\overline{\partial}_{M}u\|^{2} + \|\overline{\partial}_{M}^{*}u\|^{2}\right) + C_{\varepsilon}\|u\|_{{-1}}^{2}\;\;, u\in dom(\overline{\partial}_{M})\cap dom(\overline{\partial}_{M}^{*})\;.
\end{equation}
\end{lemma}

\begin{proof}
That (ii) and (iii) are equivalent follows from a general lemma in functional analysis that characterizes compactness of Hilbert space operators in terms of a family of estimates as in (iii) (\cite{Straube10a}, Lemma 4.3, \cite{BiardStraube16}, Lemma 7) and the fact that $L^{2}(M)$ embeds compactly into $W^{-1}(M)$.

By the same lemma, (iii) and (iii)* say that $\mathcal{P}^{k}$ is compact on $dom(\overline{\partial}_{M})\cap dom(\overline{\partial}_{M}^{*})\cap \mathcal{H}_{(p,q)}(M)^{\perp}$ and $dom(\overline{\partial}_{M})\cap dom(\overline{\partial}_{M}^{*})$, respectively. But because $\mathcal{H}_{(p,q)}(M)$ is finite dimensional (see the discussion at the end of section \ref{pre}), these two statements are equivalent.

Assume now that (iii) holds. To prove (i), it suffices to establish compactness of $\mathcal{P}^{k}G_{p,q}$ on $\mathcal{H}_{(p,q)}(M)^{\perp}$ (since $G_{p,q} \equiv 0$ on $\mathcal{H}_{(p,q)}(M)$). For such $u$, $G_{p,q}u$ is also in $\mathcal{H}_{(p,q)}(M)^{\perp}$, so that we may apply \eqref{microcompest} to $G_{p,q}u$. This gives
\begin{multline}\label{compactness}
 \|\mathcal{P}^{k}G_{p,q}u\|^{2}  \leq \varepsilon\left(\|\overline{\partial}_{M}G_{p,q}u\|^{2}+\|\overline{\partial}_{M}^{*}G_{p,q}u\|^{2}\right) + C_{\varepsilon}\|G_{p,q}u\|_{-1}^{2}\\
 = \varepsilon(\Box_{p,q}G_{p,q}u,u) + C_{\varepsilon}\|G_{p,q}u\|_{-1}^{2}
 = \varepsilon\|u\|^{2} + C_{\varepsilon}\|G_{p,q}u\|_{-1}^{2}\;.
\end{multline}
In the first equality on the second line we have used \eqref{quadform}. As $\varepsilon>0$ is arbitrary, the lemma used in the previous paragraph now shows that $\mathcal{P}^{k}G_{p,q}$ is compact on $\mathcal{H}_{(p,q)}(M)^{\perp}$,
because $u\rightarrow G_{p,q}u$ is a compact operator from $L^{2}_{(p,q)}(M)$ to $W^{-1}_{(p,q)}(M)$.

To see that (i) implies (ii), first note that on $\mathcal{H}_{(p,q)}(M)^{\perp}$, $G_{p,q} = j_{p,q}(j_{p,q})^{*}$ (\cite{BiardStraube16}, Lemma 4). Therefore, $\mathcal{P}^{k}j_{p,q}(j_{p,q})^{*}$ is compact on $\mathcal{H}_{(p,q)}(M)^{\perp}$. For the purposes of (ii), we may view $\mathcal{P}^{k}$ as a continuous operator from $\mathcal{H}_{(p,q)}(M)^{\perp}$ into $L^{2}_{(p,q)}(M)$. Denote by $(\mathcal{P}^{k})^{*}$ the adjoint of this operator.
Then $\mathcal{P}^{k}j_{p,q}(j_{p,q})^{*}(\mathcal{P}^{k})^{*} = (\mathcal{P}^{k}j_{p,q})(\mathcal{P}^{k}j_{p,q})^{*}$ is also compact. But the latter operator is compact (if and) only if $\mathcal{P}^{k}j_{p,q}$ is compact, i.e. (ii) holds.
\end{proof}

\medskip

We are now ready to formulate and prove the main result of this section. We are not considering the Green operators in the exceptional cases $q=0, (m-1)$, whence the restrictions on the range of $q$.
\begin{theorem}\label{percolation}
Let $M$ be a smooth compact pseudoconvex orientable CR-submanifold of $\mathbb{C}^n$ of hypersurface type, of CR-dimension $m-1$, let $0\leq p\leq m$. We have:

(i) if $\mathcal{P}^{+}G_{p,q}$ is compact, then so is $\mathcal{P}^{+}G_{p,q+1}$, $1\leq q\leq (m-3)$.

(ii) if $\mathcal{P}^{-}G_{p,q}$ is compact, then so is $\mathcal{P}^{-}G_{p,q-1}$, $2\leq q \leq (m-2)$. 

\end{theorem}

Because $\Box_{p,q}$ is (microlocally) elliptic on the support of $\chi^{0}$,
$\mathcal{P}^{0}G_{p,q}$ gains two derivatives, and so is in particular compact on $L^{2}_{(p,q)}(M)$. Theorem \ref{percolation} therefore immediately implies `interpolation', as follows.
\begin{corollary}\label{interpolation}
 Let $M$ be as in Theorem \ref{percolation}, let $1\leq q_{1}\leq q_{2}\leq (m-2)$. If $G_{p,q_{1}}$ and $G_{p,q_{2}}$ are compact, then so is $G_{p,r}$ for $q_{1}\leq r\leq q_{2}$.
\end{corollary}

\begin{proof}[Proof of Theorem \ref{percolation}]
We begin with (i). In view of Lemma \ref{microcompact}, what we have to show is that \eqref{microcompest2} for $(p,q)$--forms and $k=+$ implies \eqref{microcompest2} for $(p,q+1)$--forms and $k=+$. In doing so, we only have to establish \eqref{microcompest2} for smooth $(p,q+1)$--forms $u$, as they are dense in $dom(\overline{\partial}_{M}) \cap dom (\overline{\partial}_{M}^{*})$, by the Friedrichs Lemma; compare the remark at the beginning of the proof of Proposition \ref{Apqproperties}.

Via a standard partition of unity argument, it suffices to establish \eqref{microcompest2} for $(p,q+1)$--forms with support in a special boundary chart: 
$u=\sideset{}{'}\sum_{\vert I\vert=p,\vert J\vert=q+1}u_{IJ} \omega_{I}\wedge\overline{\omega_J}$. The first step is identical to the one in the proof that compactness in the $\overline{\partial}$-Neumann problem percolates up the $\overline{\partial}$-complex (compare \cite{Straube10a}, proof of Proposition 4.5, and the references there). For $1\leq k\leq (m-1)$, we build $(p,q)$-forms $v_k$ from $u$ as follows
\begin{equation}\label{v_k}
v_k=\sideset{}{'}\sum_{\vert I\vert=p,\vert K\vert=q}u_{IkK} \omega_{I}\wedge\overline{\omega_K}, 
\end{equation} where $(kK)=(k, k_1,\dots, k_q)$.
Since $\mathcal{P}^+$ acts coefficient wise, we obtain
$$\mathcal{P}^+v_k=\sideset{}{'}\sum_{|I|=p,\vert K\vert=q}\mathcal{P}^+u_{IkK}\omega_{I}\wedge\overline{\omega_K}\;.$$
Observe that 
\begin{eqnarray}\label{P+norm}
\Vert \mathcal{P}^+ u\Vert^2= \dfrac{1}{q+1} \sum_{k=1}^{m-1}\Vert \mathcal{P}^+ v_k\Vert^2, 
\end{eqnarray}
where $\dfrac{1}{q+1}$ comes from the fact that each $J$ appears $(q+1)$-times when the tuples $kK$ are put into increasing order.

\eqref{P+norm} together with the assumption (i.e. \eqref{microcompest} for $(p,q)$--forms) suggests to estimate $\overline{\partial}_{M}v_{k}$ and $\bar\partial_M^* v_k$ in terms of quantities involving $\overline{\partial}_{M}u$, $\overline{\partial}_{M}^{*}u$, and $u$. However, in contrast to the situation with the $\overline{\partial}$-complex (see \cite{Straube10a}, proof of Proposition 4.5), this strategy does not work here. The reason is that while $\overline{\partial}_{M}^{*}v_{k}$ is easily related to $\overline{\partial}_{M}^{*}u$, the same is not true for $\overline{\partial}_{M}v_{k}$ and $\overline{\partial}_{M}u$ (see below). 

In order to address this difficulty, we first notice that $\mathcal{P}^{+}$ is essentially a projection, i.e. $\|(\mathcal{P}^{+})^{2}v_{k} - \mathcal{P}^{+}v_{k}\| \lesssim \|(\mathcal{P}^{+})^{2}u - \mathcal{P}^{+}u\|$ (because $\mathcal{P}^{+}$ acts coefficientwise) is under control. The reason is that $(\chi^{+})^{2} - \chi^{+}$ is supported on a cone that stays away from the $\tau$--axis, so that one can invoke ellipticity. More precisely, \eqref{micro} gives for $(\mathcal{P}^{+})^{2}u$
\begin{multline}\label{P+2}
(\mathcal{P}^{+})^{2}u = \sum_{j,s=1}^{l}\mathcal{P}^{+}_{j}\varphi_{j}\mathcal{P}_{s}^{+}\varphi_{s}u = \sum_{j,s=1}^{l}\chi_{j}\mathcal{F}^{-1}\chi^{+}\mathcal{F}\varphi_{j}\chi_{s}\mathcal{F}^{-1}\chi^{+}\mathcal{F}\varphi_{s}u \\
=\sum_{j=1}^{l}\chi_{j}\mathcal{F}^{-1}(\chi^{+})^{2}\mathcal{F}\varphi_{j}u + \sum_{j,s=1}^{l}\chi_{j}\mathcal{F}^{-1}\chi^{+}\mathcal{F}\left[\varphi_{j}\chi_{s},\mathcal{F}^{-1}\chi^{+}\mathcal{F}\right]\varphi_{s}u \;.
\end{multline}
To obtain the third equality, we have first commuted (the multiplication operators) $\varphi_{j}\chi_{s}$ and $\mathcal{F}^{-1}\chi^{+}\mathcal{F}$, then
used that $\varphi_{j}\chi_{s}\varphi_{s}v_{k}=\varphi_{j}\varphi_{s}v_{k}$, and finally that $\sum_{s=1}^{l}\varphi_{s}v_{k}=v_{k}$.\footnote{Some care is required here. The multiplication operator $\varphi_{j}\chi_{s}$ really means multiplication on $\mathbb{R}^{2m-1}$ by the push forward of $\chi_{s}$ under the coordinate map in patch number $s$, followed by pulling back to $M$, multiplying by $\varphi_{j}$ and then pushing forward to $\mathbb{R}^{2m-1}$ under the coordinate map in patch number $j$.} Note that the commutator on the right-hand side of \eqref{P+2} commutes two operators of order zero, so is of order $-1$. Its $L^{2}$--norm is therefore dominated by $\|v_{k}\|_{-1}$, hence by $\|u\|_{-1}$ (and so is benign; only the calculus for the basic symbol classes denoted by $S^{m}$ in \cite{Stein93}, Chapter VI is needed here). \eqref{P+2} now gives
\begin{equation}\label{difference}
 (\mathcal{P}^{+})^{2}u - \mathcal{P}^{+}u = \sum_{j=1}^{l}\chi_{j}\mathcal{F}^{-1}\left((\chi^{+})^{2} - \chi^{+}\right)\mathcal{F}\varphi_{j}u + \mathcal{O}(\|u\|_{-1}) \;.
\end{equation}
Using now that $(\chi^{+})^{2} - \chi^{+}$ is supported on a cone away from the $\tau$--axis, we can invoke microlocal ellipticity of $\overline{\partial}_{M}\oplus\overline{\partial}_{M}^{*}$ on this cone (\cite{Kohn85}, estimate (2.9), \cite{Raich10}, Lemma 4.10, \cite{Nicoara06}, Lemma 4.18):
\begin{multline}\label{ellipticest}
 \|(\mathcal{P}^{+})^{2}u - \mathcal{P}^{+}u\|_{1}^{2} \lesssim \|\overline{\partial}_{M}\left((\mathcal{P}^{+})^{2} - \mathcal{P}^{+}\right)u\|^{2} + \|\overline{\partial}_{M}^{*}\left((\mathcal{P}^{+})^{2} - \mathcal{P}^{+}\right)u\|^{2} \\
 \lesssim \|\overline{\partial}_{M}u\|^{2} + \|\overline{\partial}_{M}^{*}u\|^{2} + \|u\|^{2}
 \lesssim \|\overline{\partial}_{M}u\|^{2} + \|\overline{\partial}_{M}^{*}u\|^{2} + \|u\|_{-1}^{2} \;.
\end{multline}
Here, we have commuted $\overline{\partial}_{M}$ and $\overline{\partial}_{M}^{*}$ with $\left((\mathcal{P}^{+})^{2} - \mathcal{P}^{+}\right)$ and used that these commutators are operators of order zero (see again \cite{Stein93}, Chapter VI) to obtain the second inequality. The third inequality is from \eqref{basicL2}. Using \eqref{ellipticest} and the standard interpolation inequality for Sobolev norms, we find that for all $\varepsilon >0$, there is a constant $C_{\varepsilon}$ such that (for $1\leq k\leq (m-1)$)
\begin{equation}\label{difference}
\|(\mathcal{P}^{+})^{2}v_{k} - \mathcal{P}^{+}v_{k}\| \lesssim \|(\mathcal{P}^{+})^{2}u - \mathcal{P}^{+}u\| \lesssim \varepsilon\left(\|\overline{\partial}_{M}u\| + \|\overline{\partial}_{M}^{*}u\|\right) + C_{\varepsilon}\|u\|_{-1}   \;. 
\end{equation}

\eqref{P+norm} and \eqref{difference} show that in order to obtain to desired estimate for $\mathcal{P}^{+}u$, it suffices to estimate $\|(\mathcal{P}^{+})^{2}v_{k}\|$, $1\leq k\leq (m-1)$. This fact will let us work around the difficulty mentioned above.

Inserting the $(p,q)$--form $(\mathcal{P}^{+})^{2}v_{k}$ into \eqref{microcompest2} gives
\begin{multline}\label{5}
\|(\mathcal{P}^{+})^{2}v_{k}\|^{2} \leq \varepsilon\left(\|\overline{\partial}_{M}(\mathcal{P}^{+}v_{k})\|^{2} + \|\overline{\partial}_{M}^{*}(\mathcal{P}^{+}v_{k})\|^{2}\right) + C_{\varepsilon}\|\mathcal{P}^{+}v_{k}\|_{-1}^{2} \\
\leq \varepsilon\left(\|\overline{\partial}_{M}(\mathcal{P}^{+}v_{k})\|^{2} + \|\mathcal{P}^{+}\overline{\partial}_{M}^{*}v_{k}\|^{2} + \|v_{k}\|^{2}\right) + C_{\varepsilon}\|\mathcal{P}^{+}v_{k}\|_{-1}^{2} \;.
\end{multline}
In the second inequality, we have commuted $\mathcal{P}^{+}$ with $\overline{\partial}_{M}^{*}$, and used that the commutator is an operator of order zero (note that $\mathcal{P}^{+}$ is scalar). 

We first look at $\overline{\partial}_{M}^{*}v_{k}$. This part again follows \cite{Straube10a}, p.79--80. For $\alpha=\sideset{}{'}\sum_{|I|=p,\vert K\vert =q} a_{IK} \omega_{I}\wedge\overline{\omega_K} \in L^{2}_{(p,q)}(M)$, we have
\begin{multline}\label{1}\
(\overline{\omega_k}\wedge \alpha, u)_{L^2_{(p,q+1)}(M)} =\left(\sideset{}{'}\sum_{|I|=p,\vert K\vert =q} a_{IK} \overline{\omega_k}\wedge\omega_{I}\wedge\overline{\omega_K}, \sideset{}{'}\sum_{|I|=p,\vert J\vert=q+1}u_{IJ}\omega_{I}\wedge\overline{\omega_J}\right)_{L^2_{(p,q+1)}(M)} \\= (-1)^{p}\sideset{}{'}\sum_{|I|=p,\vert K\vert=q}\int_{M} a_{IK}\overline{u_{IkK}} \; dV_M \;=\;
(-1)^{p}(\alpha, v_k)_{L^2_{(p,q)}(M)}.
\end{multline}
\eqref{1} expresses an inner product with $v_{k}$ in terms of an inner product with $u$ (read from right to left). We use this expression to estimate $\|\overline{\partial}_{M}^{*}v_{k}\|$.
Let $\beta$ be a $(p,q-1)$-form; we have
\begin{multline}\label{2}
(\bar\partial_M\beta, v_k)_{L^2_{(p,q)}(M)} = (\overline{\omega_k}\wedge\bar\partial_M\beta, u)_{L^2_{(p,q+1)}(M)} \\
= -(\bar\partial_M(\overline{\omega_k}\wedge \beta), u)_{L^2_{(p,q+1)}(M)}+ (\bar\partial_M \overline{\omega_k}\wedge \beta, u)_{L^2_{(p,q+1)}(M)},\\
= -(\overline{\omega_k}\wedge \beta, \bar\partial^*_M u)_{L^2_{(p,q)}(M)}+ (\bar\partial_M \overline{\omega_k}\wedge \beta, u)_{L^2_{(p,q+1)}(M)}\;.
\end{multline}
\eqref{2} shows that $v_{k} \in dom(\overline{\partial}_{M}^{*})$ (since it shows in particular that $\left|(\bar\partial_M\beta, v_k)_{L^2_{(p,q)}(M)}\right| \lesssim \|\beta\|$), and that moreover (since $\mathcal{P}^{+}$ is continuous on $L^{2}$)
\begin{equation}\label{adjv}
\|\mathcal{P}^{+}\overline{\partial}_{M}^{*}v_{k}\|^{2} \lesssim \Vert \bar\partial^*_M v_k\Vert^2 \lesssim \Vert \bar\partial_M^* u\Vert^2 + \Vert u\Vert^2\;;\; 1\leq k\leq (m-1)\;. 
\end{equation}
The reader should note that so far, the passage to $(\mathcal{P}^{+})^{2}v_{k}$ would indeed not have been necessary, as we have estimated $\overline{\partial}_{M}^{*}v_{k}$; having $\mathcal{P}^{+}$ in front of $\overline{\partial}_{M}^{*}$ was not necessary. It is, however, for estimating the term $\Vert \bar\partial_M(P^+ v_k)\Vert$ in \eqref{5}.

We now estimate this term. From
\begin{equation}\label{6}
\bar\partial_M(\mathcal{P}^{+}v_k)=\bar\partial_M\left(\sideset{}{'}\sum_{|I|=p,\vert K\vert=q} (\mathcal{P}^{+}u_{IkK})\,\omega_{I}\wedge\overline{\omega_K}\right) 
\end{equation}
and \eqref{localexpress}, we deduce
\begin{equation}\label{dbarPvk}
\Vert \bar\partial_M (\mathcal{P}^+v_k)\Vert^2\lesssim \sum_{j=1}^{m-1}\sideset{}{'}\sum_{|I|=p,\vert J\vert=q+1} \Vert \overline{L_j}\mathcal{P}^+(u_{IJ})\Vert^2+\Vert \mathcal{P}^+u\Vert^2.
\end{equation}
Of course, in the right-hand side of \eqref{dbarPvk}, the summation should be over a set of patches where local bases $L_{1}, \cdots, L_{m-1}$ are defined. However, that presents no problem; in fact, in view of \eqref{micro}, it suffices to estimate the right-hand side of \eqref{dbarPvk} with $\mathcal{P}^{+}$ replaced by $\mathcal{P}^{+}_{s}$, $1\leq s\leq l$. We can then assume that the open set $U^{\prime}_{s}$ is small enough so that we have such a basis. It turns out that the $\overline{L_{j}}$--derivatives of $\mathcal{P}^{+}_{s}u$ on the right-hand side of \eqref{dbarPvk} can be estimated by 
$\|\overline{\partial}_{M}(\mathcal{P}^{+}_{s}u)\| + \|\overline{\partial}_{M}^{*}(\mathcal{P}^{+}_{s}u)\|$ (plus the benign term $\|\mathcal{P}^{+}_{s}u\|$); this is the crux of the matter.

To obtain this estimate, as well as the corresponding one in the proof of (ii), we follow ideas and computations from \cite{Shaw85a, Kohn85, Nicoara06, Ahn07, RaichStraube08, Raich10, HarringtonRaich10, MunasingheStraube12}.
We start from the usual formula, obtained from integration by parts (see for example the proof of Theorem 8.3.5 in \cite{ChenShaw01}): 
\begin{multline}\label{usualformula}
\Vert \bar\partial_M(\mathcal{P}^{+}_{s}u)\Vert^2+\Vert \bar\partial_M^*(\mathcal{P}^{+}_{s}u)\Vert^2 \\  
= \sum_{j=1}^{m-1}\sideset{}{'}\sum_{|I|=p,\vert J\vert=q+1}\Vert \overline{L_j}(\mathcal{P}^{+}_{s}u_{IJ})\Vert^2
+ \sideset{}{'}\sum_{|I|=p,\vert K\vert=q}\sum_{j,k=1}^{m-1}\left([L_j,\overline{L_k}]\mathcal{P}^{+}_{s} u_{IjK}, \mathcal{P}^{+}_{s}u_{IkK}\right)_{L^2_{(p,q+1)}(M)} \\
 + \mathcal{O}\left(\| \mathcal{P}^{+}_{s}u\|\left(\| \bar{L}(\mathcal{P}^{+}_{s}u)\| +\| L(\mathcal{P}^{+}_{s}u)\|\right)+\|\mathcal{P}^{+}_{s}u\|^2\right)\; ,
\end{multline}
where $\displaystyle{\Vert \overline{L}(\mathcal{P}^{+}_{s}u)\Vert^2=\sum_{j=1}^{m-1}\sideset{}{'}\sum_{|I|=p,\vert J\vert = q+1}\| \overline{L_j}(\mathcal{P}^{+}_{s}u_{IJ})\|^2}$. 

The term $\mathcal{O}\left(\Vert\mathcal{P}^{+}_{s}u\Vert\Vert {L}(\mathcal{P}^{+}_{s}u)\Vert\right)$ is of the form $\left(g\mathcal{P}^{+}_{s}u, L(\mathcal{P}^{+}_{s})u\right)_{L^2(M)}$ where $g$ is a smooth function. With integration by parts, it can be estimated by $\mathcal{O}\left(\Vert\mathcal{P}^{+}_{s}u\Vert \Vert \overline{L}(\mathcal{P}^{+}_{s}u)\Vert +\Vert\mathcal{P}^{+}_{s}u\Vert^2\right)$. As a result, the last term on the right-hand side of \eqref{usualformula} is $\mathcal{O}\left(\Vert\mathcal{P}^{+}_{s}u\Vert \Vert \overline{L}(\mathcal{P}^{+}_{s}u)\Vert +\Vert\mathcal{P}^{+}_{s}u\Vert^2\right)$.

Denote $(c_{jk})$ the matrix of the Levi form in the basis $L_1, \dots, L_{m-1}$. Then 
\begin{equation}
[L_j,\overline{L_k}]=c_{jk} T\quad\text{mod}\; T^{1,0}M\oplus T^{0,1}M\;.
\end{equation}
Plugging this expression for the commutators into \eqref{usualformula}, taking real parts, and using the observation from the preceding paragraph gives 
\begin{multline}\label{Re}
\Vert \bar\partial_M (\mathcal{P}^{+}_{s}u)\Vert^2+\Vert \bar\partial_M^* (\mathcal{P}^{+}_{s}u)\Vert^2 \\
= \sum_{j=1}^{m-1}\sideset{}{'}\sum_{|I|=p,\vert J\vert=q+1}\Vert \overline{L_j} (\mathcal{P}^{+}_{s}u_{IJ})\Vert^2 
+ \text{Re}\left( \sideset{}{'}\sum_{|I|=p,\vert K\vert=q}\sum_{j,k=1}^{m-1} (c_{jk}T (\mathcal{P}^{+}_{s}u_{IjK}),\mathcal{P}^{+}_{s} u_{IkK})\right) \\
+ \mathcal{O}\left(\Vert \mathcal{P}^{+}_{s}u\Vert \Vert \bar{L}\mathcal{P}^{+}_{s}u\Vert + \Vert \mathcal{P}^{+}_{s}u\Vert^2\right)\;.
\end{multline}
Estimating the contribution $\|\mathcal{P}^{+}_{s}u\|\|\overline{L}\mathcal{P}^{+}_{s}u\|$ to the error term by $(l.c.)\|\mathcal{P}^{+}_{s}u\|^{2} + (s.c)\|\overline{L}\mathcal{P}^{+}_{s}u\|^{2}$ and absorbing $(s.c)\|\overline{L}\mathcal{P}^{+}_{s}u\|^{2}$ into the first sum on the second line of \eqref{Re} shows that in the next estimate, we only need to have an error term $\mathcal{O}(\|\mathcal{P}^{+}_{s}u\|^{2})\lesssim\mathcal{O}(\|u\|^{2})$. Moreover, $T\mathcal{P}^{+}_{s}u_{IjK} = T\chi_{s}\mathcal{F}^{-1}\chi^{+}\mathcal{F}u_{IjK} = \chi_{s}T\mathcal{F}^{-1}\chi^{+}\mathcal{F}u_{IjK}+ \mathcal{O}(\|u\|) = \chi_{s}\mathcal{F}^{-1}\tau^{+}\chi^{+}\mathcal{F}u_{IjK}+ \mathcal{O}(\|u\|) $. We have used here that $[T,\chi_{s}]$ is of order zero, $T\mathcal{F}^{-1}=\mathcal{F}^{-1}\tau$, and that on the support of $\chi^{+}$, $\tau = \tau^{+}$. Inserting these observations into \eqref{Re}, we obtain
\begin{multline}\label{Re2}
\Vert\bar\partial_M(\mathcal{P}^{+}_{s}u)\Vert^2+\Vert \bar\partial_M^*(\mathcal{P}^{+}_{s}u)\Vert^2   \gtrsim \sum_{j=1}^{m-1}\sideset{}{'}\sum_{|I|=p,\vert J\vert=q+1}\Vert \overline{L_j}(\mathcal{P}^{+}_{s}u_{IJ})\Vert^2 \\
+  \text{Re}\left(\sideset{}{'}\sum_{|I|=p,\vert K\vert=q}\sum_{j,k=1}^{m-1} (c_{jk}\chi_{s}\mathcal{F}^{-1}\tau^+\chi^+ \mathcal{F}u_{IjK}, \chi_{s}\mathcal{F}^{-1}\chi^+\mathcal{F}u_{IkK})\right)
 + \mathcal{O}(\Vert u\Vert^2)\;.
\end{multline}
$M$ is pseudoconvex, that is, the matrix $(c_{jk})$ is positive semi definite. Therefore,
we can now apply G\aa rding's inequality\footnote{Alternatively, we can write $\tau^{+}=(\sqrt{\tau^{+}})(\sqrt{\tau^{+}})$. Now commute one factor $\sqrt{\tau^{+}}$ all the way to the left of $c_{jk}$, then move it over to the right-hand side in the inner product and commute it past $\mathcal{F}^{-1}$; the resulting error terms are benign. But now the innermost sum is of the form $\sum_{j,k=1}^{m-1}c_{jk}b_{j}\overline{b_{k}}$, and so is manifestly non negative. Our desired estimate follows. Strictly speaking, we use a smooth function that agrees with $\sqrt{\tau^{+}}$ on the support of $\chi^{+}$, so that there are no issues with the pseudodifferential calculus. Compare also \cite{Kohn02}, p.222.} (see for example \cite{Kohn02}, Lemma 2.5, \cite{LaxNirenberg66}, Theorems 3.1, 3.2) to the term with the real part in \eqref{Re2} to obtain 
\begin{equation}\label{Garding}
\text{Re}\left(\sideset{}{'}\sum_{|I|=p,\vert K\vert=q}\sum_{j,k=1}^{m-1} (c_{jk}\chi_{s}\mathcal{F}^{-1}\tau^+\chi^+ \mathcal{F}u_{IjK}, \chi_{s}\mathcal{F}^{-1}\chi^+\mathcal{F}u_{IkK})\right) \gtrsim -\Vert u\Vert^2\;.
\end{equation}
Note that here it is crucial to have $\tau^{+}$ instead of $\tau$ in order to have the positivity required for G\aa rding's inequality. Combining \eqref{Garding} with \eqref{Re2}, we obtain 
\begin{equation}\label{7}
\Vert \overline{L}\mathcal{P}^+u\Vert^2\lesssim \Vert \bar\partial_M u\Vert^2+\Vert \bar\partial_M^* u\Vert^2 + \Vert u\Vert^2\;.
\end{equation}

Finally, combining \eqref{7} with \eqref{dbarPvk}, inserting the result together with \eqref{adjv} into \eqref{5} gives the desired estimate for $\|(\mathcal{P}^{+})^{2}v_{k}\|^{2}$. We have also used that $\|\mathcal{P}^{+}v_{k}\|_{-1} \lesssim \|v_{k}\|_{-1} \lesssim \|u\|_{-1}$ (since $\mathcal{P}^{+}$ is of order zero), and that $\|v_{k}\|^{2} \lesssim \|u\|^{2} \lesssim \|\overline{\partial}_{M}u\|^{2} + \|\overline{\partial}_{M}^{*}u\|^{2} + \|u\|_{-1}^{2}$ (this is a weak special case of \eqref{basicL2}). This completes the proof of part (i) of Theorem \ref{percolation}.
\medskip 

The proof of $(ii)$ of Theorem 2 is similar, but there is an additional twist. Again using Lemma \ref{microcompact}, we assume that \eqref{microcompest2} holds for $(p,q)$--forms, with $k=-$; we must show that it then also holds for $(p,q-1)$--forms. Let $u=\sideset{}{'}\sum_{|I|=p,\vert J\vert=q-1}u_{IJ}\,\omega_{I}\wedge\overline{\omega_{J}}$ be a $(p,q-1)$-form. As in the proof of (i), we may assume that $u$ is smooth. We build a $(p,q)$-form from $u$ as follows:
\begin{equation}\label{down}
 v_k=\sideset{}{'}\sum_{|I|=p,\vert J\vert = q-1} u_{IJ}\,\overline{\omega_{k}}\wedge \omega_{I}\wedge\overline{\omega_J}= \overline{\omega_k}\wedge u\;.
\end{equation}

Then $\displaystyle{\Vert \mathcal{P}^- u\Vert^2\lesssim\sum_{k=1}^{m-1}\Vert \mathcal{P}^- v_k\Vert^2}$, and we have to estimate $\|\mathcal{P}^{-}v_{k}\|$, $1\leq k\leq (m-1)$. The argument that it suffices to estimate $\|(\mathcal{P}^{-})^{2}v_{k}\|$, $1\leq k\leq (m-1)$ instead is the same as in the proof of part (i). Inserting $(\mathcal{P}^{-})^{2}v_{k}$ into \eqref{microcompest2} shows that we need to control $\|\overline{\partial}_{M}(\mathcal{P}^{-}v_{k})\|$ and $\|\overline{\partial}_{M}^{*}(\mathcal{P}^{-}v_{k})\|$ in terms of $\|\overline{\partial}_{M}u\|$, $\|\overline{\partial}_{M}^{*}u\|$, and $\|u\|$. This time, it is the $\overline{\partial}_{M}$--term that is straightforward:
\begin{multline}
\;\;\;\;\;\|\overline{\partial}_{M}\left(\mathcal{P}^{-}v_{k}\right)\| \lesssim \|\mathcal{P}^{-}\overline{\partial}_{M}v_{k}\| + \|v_{k}\| \\
 \lesssim \|\overline{\partial}_{M}\left(\overline{\omega_{k}}\wedge u\right)\| + \|\overline{\omega_{k}}\wedge u\|
 \lesssim \|\overline{\partial}_{M}u\| + \|u\|\;.\;\;\;\;\;
\end{multline}

\eqref{adjoint} gives 
\begin{multline}\label{8}
 \left\|\overline{\partial}_{M}^{*}(\mathcal{P}^{-}v_{k})\right\| \lesssim \sum_{j=1}^{m-1}
 \sideset{}{'}\sum_{|I|=p,|K|=q-1}\|L_{j}(\mathcal{P}^{-}v_{k})_{IjK}\| + \|\mathcal{P}^{-}v_{k}\| \\
 \lesssim \sum_{j=1}^{m-1}
 \sideset{}{'}\sum_{|I|=p,|J|=q-1}\|L_{j}(\mathcal{P}^{-}u_{IJ})\| + \|u\|\;.
\end{multline}
As above, it suffices to estimate $\|L_{j}(\mathcal{P}^{-}_{s}u_{IJ})\|$, $1\leq s\leq l$.
We want to estimate the terms on the right-hand side of \eqref{8} using \eqref{usualformula}, with $\mathcal{P}^{-}_{s}$ in the place of $\mathcal{P}^{+}_{s}$. In order to do so, we first need to integrate by parts to obtain $\overline{L_{j}}$--terms:
\begin{multline}\label{switch}
\left(L_{j}(\mathcal{P}^{-}_{s}u_{IJ}), L_{j}(\mathcal{P}^{-}_{s}u_{IJ})\right) 
= -\left(\mathcal{P}^{-}_{s}u_{IJ}, \overline{L_{j}}L_{j}\mathcal{P}^{-}_{s}u_{IJ}\right) + \mathcal{O}\left(\|\mathcal{P}^{-}_{s}u_{IJ}\|\|L_{j}(\mathcal{P}^{-}_{s}u_{IJ})\|\right) \\
= - \left(\mathcal{P}^{-}_{s}u_{IJ}, L_{j}\overline{L_{j}}\mathcal{P}^{-}_{s}u_{IJ}\right) +
\left(\mathcal{P}^{-}_{s}u_{IJ}, \left[L_{j}, \overline{L_{j}}\right]\mathcal{P}^{-}_{s}u_{IJ}\right) +
\mathcal{O}\left(\|\mathcal{P}^{-}_{s}u_{IJ}\|\|L_{j}(\mathcal{P}^{-}_{s}u_{IJ})\|\right) \\
= \|\overline{L_{j}}\mathcal{P}^{-}_{s}u_{IJ}\|^{2} + \left(\mathcal{P}^{-}_{s}u_{IJ}, c_{jj}T\mathcal{P}^{-}_{s}u_{IJ}\right) + \mathcal{O}\left(\|\mathcal{P}^{-}_{s}u\|\|L_{j}(\mathcal{P}^{-}_{s}u_{IJ})\| + \|\mathcal{P}^{-}_{s}u_{IJ}\|^{2}\right)\;.
\end{multline}
As in \eqref{usualformula}, we have used here that the undifferentiated term in the second integration by parts can itself be integrated by parts to move the $\overline{L_{j}}$--term to the left as an $L_{j}$--term. Solving this expression for $\|\overline{L_{j}}\mathcal{P}^{-}_{s}u_{IJ}\|^{2}$, substituting the result into (the analogue, for $\mathcal{P}^{-}_{s}u_{IJ}$, of) \eqref{usualformula}, absorbing terms, and taking real parts leads to the analogue of \eqref{Re}. 
\begin{multline}\label{9}
\Vert \bar\partial_M (\mathcal{P}^{-}_{s}u)\Vert^2+\Vert \bar\partial_M^* (\mathcal{P}^{-}_{s}u)\Vert^2 
= \sum_{j=1}^{m-1}\sideset{}{'}\sum_{|I|=p,\vert J\vert=q-1}\left(\Vert L_j (\mathcal{P}^{-}_{s}u_{IJ})\Vert^2 - \text{Re}\left(c_{jj}T\mathcal{P}^{-}_{s}u_{IJ}, \mathcal{P}^{-}_{s}u_{IJ}\right)\right) \\
+ \text{Re}\left( \sideset{}{'}\sum_{|I|=p,\vert K\vert=q-2}\sum_{j,k=1}^{m-1} (c_{jk}T (\mathcal{P}^{-}_{s}u_{IjK}),\mathcal{P}^{-}_{s}u_{IkK})\right) 
+ \mathcal{O}\left(\Vert \mathcal{P}^{-}_{s}u\Vert \Vert \bar{L}\mathcal{P}^{-}_{s}u\Vert + \Vert \mathcal{P}^{-}_{s}u\Vert^2\right)\;.
\end{multline}
We have used that $\text{Re}\left(T\mathcal{P}^{-}_{s}u_{IJ}, \mathcal{P}^{-}_{s}u_{IJ}\right) = \text{Re}\left(\mathcal{P}^{-}_{s}u_{IJ}, T\mathcal{P}^{-}_{s}u_{IJ}\right)$ (on the right-hand side of the first line). Observe that
\begin{multline}
\sum_{j=1}^{m-1}\sideset{}{'}\sum_{|I|=p,\vert J\vert=q-1}\text{Re}\left(c_{jj}T\mathcal{P}^{-}_{s}u_{IJ}, \mathcal{P}^{-}_{s}u_{IJ}\right) \\
= \frac{1}{q-1}\text{Re}\left(\sideset{}{'}\sum_{|I|=p,\vert K\vert=q-2}\;\sum_{j,k=1}^{m-1}\left(\delta_{jk}tr(c_{rd})T\mathcal{P}^{-}_{s}u_{IjK}, \mathcal{P}^{-}_{s}u_{IkK}\right)\right)\;,
\end{multline}
where $\delta_{jk}$ is the Kronecker $\delta$, $tr$ denotes the trace, and $(c_{rd})$ is the matrix of the Levi form. The factor $\frac{1}{q-1}$ arises because each $(q-1)$--tuple $J = (j_{1}, \hdots, j_{q-1})$ arises in precisely $(q-1)$ ways as the increasingly ordered version of a tuple $jK$, namely from $j_{1}K_{1}, \hdots, j_{q-1}K_{q-1}$, where $K_{s}$ omits $j_{s}$ from $J$. Proceeding as in the paragraph following \eqref{Re} results in the following estimate:
\begin{multline}\label{Re3}
 \Vert\bar\partial_M(\mathcal{P}^{-}_{s}u)\Vert^2+\Vert \bar\partial_M^*(\mathcal{P}^{-}_{s}u)\Vert^2   \gtrsim \sum_{j=1}^{m-1}\sideset{}{'}\sum_{|I|=p,\vert J\vert=q-1}\Vert L_j(\mathcal{P}^{-}_{s}u_{IJ})\Vert^2 \\
+  \text{Re}\left(\sideset{}{'}\sum_{|I|=p,\vert K\vert=q-2}\;\sum_{j,k=1}^{m-1}\left(\left(c_{jk}-\frac{1}{q-1}\delta_{jk}tr(c_{rd})\right)\chi_{s}\mathcal{F}^{-1}\tau^-\chi^- \mathcal{F}u_{IjK}, \chi_{s}\mathcal{F}^{-1}\chi^-\mathcal{F}u_{IkK}\right)\right) \\
 + \mathcal{O}(\Vert u\Vert^2)\;.
\end{multline}
The sum of any $(q-1)$ eigenvalues of the matrix $\left(c_{jk}-\frac{1}{q-1}\delta_{jk}tr(c_{rd})\right)$ has the form `sum of $(q-1)$ eigenvalues of $(c_{jk})$ minus the trace of this matrix, and so equals minus the sum of the remaining eigenvalues. Because $M$ is pseudoconvex, minus this latter sum is less than or equal to zero. This means that the form
$\displaystyle{\sideset{}{'}\sum_{|K|=q-2} \sum_{j,k=1}^{m-1}\left(c_{jk}-\frac{1}{q-1}\delta_{jk}tr(c_{rd})\right)b_{jK}\overline{b_{kK}}\leq 0}$ for every $(0,q-1)$--form (see \cite{Straube10a}, Lemma 4.7). Because $\tau^{-}$ is also negative of the support of $\chi^{-}$, 
the Hermitian form (on $u$) inside the Re--term in \eqref{Re3}
satisfies the assumptions in G\aa rding's inequality (see for example \cite{Kohn02}, Lemma 2.5, \cite{LaxNirenberg66}, Theorems 3.1, 3.2), and we obtain the analogue of \eqref{Garding} above. Alternatively, we could argue as outlined in the footnote there. The estimate that then results for $\displaystyle{\sum_{j=1}^{m-1}\sideset{}{'}\sum_{|I|=p,\vert J\vert=q-1}\Vert L_j(\mathcal{P}^{-}_{s}u_{IJ})\Vert^2}$, together with \eqref{8}, allows to finish the proof of (ii) in the same way as that of (i). This completes the proof of Theorem \ref{percolation}.
\end{proof}

\newpage
\begin{remarks}

\medskip 

1) Although we have concentrated on compactness, our methods also give results for subellipticity. For example, to show that 
\begin{equation}\label{+subelliptic}
\|\mathcal{P}^{+}u\|_{\varepsilon}^{2} \lesssim \left(\|\overline{\partial}_{M}u\|^{2} + \|\overline{\partial}_{M}^{*}u\|^{2}\right) + \|u\|_{{-1}}^{2}\;\;, u\in dom(\overline{\partial}_{M})\cap dom(\overline{\partial}_{M}^{*})\; 
\end{equation}
(i.e. subellipticity for $\mathcal{P}^{+}$) holds for $(p,q+1)$--forms if it holds for $(p,q)$--forms, one can follow the proof of (i) in Theorem \ref{percolation} essentially verbatim, replacing $\|\cdot\|$ with $\|\cdot\|_{\varepsilon}$ in the appropriate places. The analogous remark applies to subelliptic estimates for $\mathcal{P}^{-}$.

\medskip

2) In the proofs of Theorem \ref{percolation} we can get by with assuming less than pseudoconvexity, in both (i) and (ii), when we fix the level $q$, in the usual way. For example, in the proof of (ii), right after \eqref{Re3}, we needed to know that minus the sum of the remaining $(m-q)$ eigenvalues is negative. This will be the case if we only assume that the Levi from has the property (at every point) that the sum of the smallest $(m-q)$ eigenvalues is nonnegative. A similar remark applies for the proof of (i). In addition, we do need to have the closed range properties for $\overline{\partial}_{M}$ to have (bounded) Green operators. However, for a particular level of $q$, there are sufficient conditions in terms of positivity of appropriate sums of eigenvalues as well, see `weak condition $Y_{q}$' in  \cite{HarringtonRaich10}.

\end{remarks}

\bigskip
\providecommand{\bysame}{\leavevmode\hbox to3em{\hrulefill}\thinspace}

\end{document}